\documentclass{amsproc}

\newtheorem{theorem}{Theorem}[section]
\newtheorem{lemma}[theorem]{Lemma}
\newtheorem{prop}[theorem]{Proposition}
\newtheorem{cor}[theorem]{Corollary}

\theoremstyle{definition}
\newtheorem{definition}[theorem]{Definition}

\theoremstyle{remark}
\newtheorem{remark}[theorem]{Remark}

\numberwithin{equation}{section}

\DeclareMathOperator{\rank}{rank}
\DeclareMathOperator{\Span}{Span}

\DeclareMathOperator{\End}{End}

\DeclareMathOperator{\Aut}{Aut}

\DeclareMathOperator{\Tr}{Tr}
\DeclareMathOperator{\SL}{SL}

\newcommand{\Z}{{\mathbb{Z}}}
\newcommand{\Q}{{\mathbb{Q}}}

\newcommand{\R}{{\mathbb{R}}}

\newcommand{\C}{{\mathbb{C}}}

\newcommand{\eg}{\textit{e.g. }}

\begin{document}

\title{The unreasonable effectiveness of the tensor product.}

\author{Renaud Coulangeon} 
\address{Universit\'e Bordeaux 1, 351 Cours de la Lib\'eration, Talence, France}
\email{renaud.coulangeon@math.u-bordeaux1.fr}
\author{Gabriele Nebe}
\address{
Lehrstuhl D f\"ur Mathematik, RWTH Aachen University,
52056 Aachen, Germany}
\email{nebe@math.rwth-aachen.de}

\dedicatory{This paper is dedicated to Boris Venkov.}


\maketitle

\begin{abstract}
Using the Hermitian tensor product description of the 
extremal even unimodular lattice of dimension 72  in \cite{dim72} 
we show its extremality with the methods from \cite{coul}.
\end{abstract}

\section{Introduction}

The paper \cite{dim72} describes the construction of an extremal even unimodular lattice $\Gamma $ 
of dimension 72
of which the existence was a longstanding open problem. 
There are at least three independent proofs of extremality of this lattice, 
two of them are given in \cite{dim72} and rely heavily on computations within the set of minimal
vectors of the Leech lattice. The other one is also
 highly computational and uses the 
methods of \cite{StehleWatkins}. 
All these proofs do not give much structural insight why this lattice is extremal. 
The present paper uses the construction of $\Gamma $ as a Hermitian tensor product 
to derive a more structural proof of extremality of $\Gamma $ with the
methods in \cite{coul}. Moreover, the computational complexity of this new proof is far lower than the previously known ones. 

Let $L $ be a lattice in Euclidean $\ell$-space $(\R^{\ell},x\cdot y)$.
Then the \textit{dual lattice} is $L^{*} := \{ x\in \R^{\ell} \mid x\cdot \lambda \in \Z 
\mbox{ for all } \lambda \in L \}$.
The lattice is called \textit{unimodular} (resp. \textit{modular}), if 
$L$ is equal (resp. similar to) $L^{*}$. Being (uni-)modular implies certain
invariance properties of the theta series of $L$. 
In particular the theta series of an even unimodular lattice is a 
modular form for the full modular group $\SL_2(\Z )$. The theory of 
modular forms allows to show that the \textit{minimum} 
$$\min (L):= \min \{ \lambda \cdot \lambda  \mid 0\neq \lambda  \in \L \}$$
of $L$ is bounded from above by $2+2\lfloor \frac{{\ell}}{24} \rfloor $. 
Lattices achieving equality are called \textit{extremal}.

Several examples of extremal (uni-)modular lattices obtained as
Hermitian tensor products of lower dimensional lattices were already known, 
see for instance \cite{Ba-Ne} for a construction of extremal lattices of dimension $40$ and $80$ related to the Mathieu group $M_{22}$. 
This situation is nevertheless rather exceptional.
 Briefly, in order that a tensor product $L\otimes M$ gives rise to a 
dense sphere packing, it has to contain simultaneously \textit{split}
 and \textit{non split} short vectors.
 Obviously, the minimal length of a split vector $l\otimes m$ 
is exactly $\min L \min M$ while the minimal length of a non 
split vector $\sum_{i=1}^r l_i \otimes m_i$ ($r>1$) will usually be
 strictly smaller. The challenge, when allowing non split minimal vectors,
 is thus precisely to prevent their minimal length from dropping. 

In the first section of this note, we review rather well-known results 
about the minima of tensor products of lattices over $\Z$, mainly due to Kitaoka. 
Also, and maybe less well-known, we comment on the behaviour of tensor product
 with respect to the associated sphere packing density.
 Roughly speaking, we show that the tensor product of two lattices over $\Z$ 
of small dimension cannot achieve a maximal density,
 even locally see Proposition \ref{prop1} and its corollary
 (here ``small'' means ``less than $43$''). 

In contrast, tensor product over small field extensions, \eg imaginary quadratic, 
may produce examples of dense or extremal lattices, 
among which the constructions already mentioned,
 in particular the extremal lattice $\Gamma$ in dimension $72$. 
Section 3 recalls some facts on Hermitian lattices over imaginary
quadratic number fields. These are then applied to give a construction
of one extremal even unimodular 
48-dimensional lattices as a Hermitian tensor product over $\Z[\frac{1+\sqrt{-11}}{2}] $ in Section 4 before we give a new proof of the 
extremality of $\Gamma $ in Section 5. 

\section{Tensor products over $\Z $} \label{s2} 
In this section, we analyze the behaviour of tensor product of Euclidean lattices with respect to \textit{perfection}, a notion which we first recall.

Let $L$ be a Euclidean lattice equipped with a basis $\mathcal B$. 
We denote by $S(L)$ the set of its minimal vectors 
(non zero vectors of shortest length).
 To every minimal vector $x$ we associate the integral column vector $X$ 
of its coordinates on $\mathcal B$, and denote $S_{\mathcal B}$ the set of such
 $X$s as $x$ runs through $S(L)$.
 The \textit{rank of perfection} of $L$ is the integer
\begin{equation*}
r_{perf}(L)= \dim \Span_{\R} \left \{ X X^t \ \mid \  X \in S_{\mathcal B}(L)\right \}.
\end{equation*}
Clearly, $r_{perf}(L)$ does not depend on the choice of a particular basis 
$\mathcal B$, and is at most $\dfrac{\ell(\ell+1)}{2}$, 
where $\ell=\rank L$, since $X X^{t}$ is a symmetric matrix of size $\ell$
 for all $X \in S_{\mathcal B}(L)$.
\begin{definition} A lattice $L$ of rank $\ell$ is perfect if $r_{perf}(L)=\dfrac{\ell(\ell+1)}{2}$.
\end{definition}

Lattices achieving a local maximum of density are classically called 
\textit{extreme}. Perfection is a necessary condition for a lattice to be extreme, as was first observed by Korkine and Zolotareff (see \cite[Chapter 3] {Martinet} for historical comments).


Every element of the tensor product $L\otimes_{\Z} M$ of two Euclidean lattices  can be written as a sum of \textit{split} vectors $x\otimes y$ ($x \in L$, $y \in M$). The Euclidean structure on $L\otimes_{\Z} M$ is defined, on split vectors, by the formula
\begin{equation*}
\left (x\otimes y\right ) \cdot \left (z \otimes t\right ) = \left (x\cdot z\right ) \left (y\cdot t\right )
\end{equation*}
which extends uniquely to a well-defined inner product on $L\otimes_{\Z} M$.

\begin{prop}\label{prop1} 
Let $L$ and $M$ be Euclidean lattices of rank at least $2$. If all the minimal vectors of $L\otimes_{\Z} M$ are split, then $L\otimes M$ is not perfect, and consequently not extreme.
\end{prop}
\begin{proof}
 Fix bases $\mathcal B$ and $\mathcal C$ of $L$ and $M$ respectively. Under the hypothesis that all minimal vectors of $L\otimes_{\Z} M$ are split we have
\begin{align*}
r_{perf}(L\otimes_{\Z} M)&= \dim \Span_{\R} \left \{ (X \otimes Y)(X \otimes Y)^{t} \ \mid \  X \in S_{\mathcal B}(L), \ Y \in S_{\mathcal C}(M) \right \}\\
&= \dim \Span_{\R} \left \{ X X^{t} \otimes YY^{t}\ \mid \  X \in S_{\mathcal B}(L), \ Y \in S_{\mathcal C}(M) \right \}\\
&\leq\dfrac{\ell(\ell+1)}{2}\dfrac{m(m+1)}{2} \\&< \dfrac{\ell m(\ell m+1)}{2}
\end{align*}
whence the conclusion.
\end{proof}

The question as to whether the minimal vectors of a tensor product are split or
 not has been investigated thoroughly by Kitaoka (see \cite[Chapter 7]{kitaoka}).
 Combining some of his results with the previous proposition one obtains :
\begin{cor}\label{cor}  If $\rank L \leq 43$ or $\rank M \leq 43$, then  $L\otimes M$ is not perfect, and consequently not extreme.
\end{cor}
\begin{proof}
 By \cite[Theorem 7.1.1]{kitaoka} we know that if the conditions of the corollary are satisfied, then the minimal vectors are split, whence the conclusion using Proposition \ref{prop1}
\end{proof}

\begin{remark}
\begin{enumerate}
\item To our knowledge, no explicit examples of lattices $L$ and $M$ such that $L\otimes_{\Z}M$ contains non split minimal vectors is known 
(it would require $L$ and $M$ to have rank at least $44$). 
However, it is known thanks to an unpublished theorem of Steinberg (see \cite[Theorem 9.6]{MH}) 
that in any dimension $n \geq 292$ there exist unimodular lattices $L$ and $M$ such
 that $\min L\otimes_{\Z} M < \min L \min M$ (the proof is of course non constructive). 
\item As is well-known, extremal even unimodular lattices of dimension 
$24k$ or $24k+8$ are extreme (cf. \cite{BaVe} also for the modular
analogues), hence perfect.
 Consequently, there is no hope to obtain new extremal modular lattices 
in dimension $24k$ or $24k+8$ $\leq 43^2$ as tensor product
 \textit{over $\Z$}
 of lattices in smaller dimensions.
Note that this also follows from the definition of extremality since 
for $\ell ,m\geq 8$ 
$$
(2+2\lfloor \frac{\ell }{24} \rfloor )
(2+2\lfloor \frac{m}{24} \rfloor ) < 
(2+2\lfloor \frac{\ell m}{24} \rfloor ) .$$
\end{enumerate}
\end{remark}

\section{Preliminaries on Hermitian lattices}\label{s3} 
For sake of completeness, we recall in this section some basic notation and lemmas about Hermitian lattices 
(see \cite{coul} or \cite{hoff} for complete proofs). 
Let $K$ be an imaginary quadratic field,  with ring of integers $\mathcal O_K$. 
The non trivial Galois-automorphism of $K$ is denoted by $\overline{ \phantom{a}}$
 (identified with the classical complex conjugation if an embedding of $K$ in $\C$ is fixed).
 We denote by $\mathcal D_{K\slash\Q}$
 the different of $K\slash\Q$ and $\mathfrak d_K$ its discriminant.
 A Hermitian lattice in a finite-dimensional $K$-vector space $V$, 
endowed with a positive definite Hermitian form $h$, is a finitely generated $\mathcal O_K$-submodule of $V$ 
containing a $K$-basis of $V$. The (Hermitian) dual of a Hermitian lattice $L$ is defined as
\begin{equation*}
L^{\#}= \left\lbrace y \in V \ \mid \ h(y,L) \subset \mathcal O_K\right\rbrace .
\end{equation*}
Its \textit{discriminant} $d_L$ is defined via the choice of a pseudo-basis:
 writing $L= \mathfrak a _1 e_1 \oplus \dots \oplus \mathfrak a _m e_m$,
 where $\left\lbrace e_1, \dots, e_m\right\rbrace$
 is a $K$-basis of $V \simeq K^m$ and the $\mathfrak a _i$s are fractional ideals in $K$,
 we define $d_L$ as the unique positive generator in $\Q$ of the ideal
 $\det\left( h(e_i,e_j)\right) \prod \mathfrak a _i\overline{\mathfrak a _i}$.
 This definition is independent of the choice of a pseudo-basis $\left( \mathfrak a _i, e_i\right)$
 and in the specific case where $\mathcal O_K$ is principal, one may take $\mathfrak a _i=\mathcal O_K$
 for all $i$, and $d_L$ is nothing but the determinant of the Hermitian Gram matrix
 of a basis of $L$.

As in the Euclidean case (see \cite[Proposition 1.2.9]{Martinet}) we obtain the 
following lemma. 

\begin{lemma}\label{proj} Let $L$ be a Hermitian lattice,
$F$ a $K$-subspace of $KL=V$, $p$ the orthogonal projection onto $F^{\perp} $. Then
\begin{equation}
d_L=d_{F\cap L} d_{p(L)}
\end{equation}
\end{lemma}

For any $1\leq r \leq m=\rank_{\mathcal O_K}L$ we define $d_r(L)$ 
as the minimal discriminant of a free $\mathcal O_K$-sublattice of rank $r$ of $L$.
 In particular, one has  
$$d_1(L)=\min(L):= \min \{ h(v,v) \mid 0\neq v\in L \}. $$
The minimal discriminants of $L$ and $L^{\#}$ satisfy the following symmetry relation, 
the proof of which is the same as in the Euclidean case (see \cite[Proposition 2.8.4]{Martinet}).

\begin{lemma}\label{ddd}
Let $L$ be a Hermitian lattice of rank $m$. Then, for any $1 \leq r  \leq m-1$, one has
\begin{equation}
d_L=d_{r}(L)d_{m-r}(L^{\#})^{-1}.
\end{equation}
\end{lemma}

By restriction of scalars, an $\mathcal O_K$-lattice of rank $m$ can be viewed as a $\Z$-lattice of rank $2m$, the {\em trace lattice of $L$},
 with inner product defined by
\begin{equation}\label{inner} 
x \cdot y = \Tr_{K\slash\Q} h(x,y).
\end{equation}
The dual $L^{*}$ of $L$ with respect to that inner product is linked to $L^{\#}$ by
\begin{equation}
L^{*}=\mathcal D_{K\slash\Q}^{-1}L^{\#}
\end{equation}
whence the relation
\begin{equation}
\det L = \lvert\mathfrak d_K\rvert ^m (d_L)^2.
\end{equation}

Note that, because of (\ref{inner}), the minimum of $L$, viewed as an ordinary $\Z$-lattice,
 is twice its "Hermitian" minimum $d_1(L)$. To avoid any confusion, we stick to Hermitian minima in what follows. 

For the proof of the main result, we use the technique developed in \cite{coul} to bound the minimum of a
 Hermitian tensor product. Suppose $L$ and $M$ are Hermitian lattices over a number field $K$.
 Then any vector $z\in L\otimes_{\mathcal O_K} M$ is a sum of tensors of the form 
$v\otimes w$ with $v\in L$ and $w\in M$. The minimal number of
summands in such an expression is called the {\em rank } of $z$. 
Clearly the rank of any vector is less than the minimum of the dimension of the
two tensor factors. 

As in the Euclidean case, the Hermitian structure on $L\otimes_{\mathcal O_K}  M$ is defined, on split vectors, by the formula
\begin{equation*}
h\left (x\otimes y, z \otimes t\right ) = h\left (x,z\right ) h\left (y, t\right )
\end{equation*}
which extends uniquely to a well-defined positive definite Hermitian form on $L\otimes_{\mathcal O_K}  M$.

\begin{prop} (\cite[Proposition 3.2]{coul})\label{bound} 
Let $L$ and $M$ be Hermitian lattices and denote by $d_r(L)$ the minimal 
determinant of a rank $r$ sublattice of $L$. 
Then for any vector $z\in L\otimes _{\mathcal O_K} M $ of rank $r$ 
one has
\begin{equation}\label{tens} 
h(z,z) \geq r d_r(L)^{1/r} d_r(M)^{1/r} .
\end{equation}
Moreover, a vector $z$ of rank $r$ in $ L\otimes _{\mathcal O_K} M $ for which equality holds in (\ref{tens}) exists if and only if $M$ and $L$ contain minimal $r$-sections $M_r$ and $L_r$ such that $M_r \simeq L_r^{\#}$.
\end{prop}
\begin{proof}
 The inequality (\ref{tens}) is precisely \cite[Proposition 3.2]{coul}.  The last assertion follows from close inspection of the proof, which shows that $h(z,z) =r d_r(L)^{1/r} d_r(M)^{1/r}$ if and only if $z=\sum_{i=1}^r e_i \otimes f_i$ where $\{e_1, \dots, e_r\}$, resp. $\{f_1, \dots, f_r\}$, are $\mathcal O_K$-bases of minimal sections $M_r$ and $L_r$ of $M$ and $L$ respectively, such that $\left( h(e_i,e_j)\right) _{i,j}=\overline{\left( h(f_i,f_j)\right)} _{i,j}^{-1}$.
\end{proof}

\subsection{Two dimensional Hermitian lattices.} \label{eukl}

The results in this section are certainly well known, we include them together with the 
short proof for completeness.

\begin{definition}
The {\em Euclidean minimum} of $\mathcal O_K$ is defined as 
$$\mu (\mathcal O_K) := \sup _{x\in K} \inf _{a\in \mathcal O _K} N_{K/\Q } (x-a ) .$$
An element $z\in K$ such that $N(z) = \inf _{a\in \mathcal O _K} N_{K/\Q } (z-a ) = \mu (\mathcal O_K) $
is called a {\em deep hole} of $O_K$.
\end{definition}

Note that the Euclidean minimum is just the covering radius of the 
lattice $\mathcal O_K$ with respect to the positive definite bilinear form 
$x\cdot y:= \frac{1}{2} \Tr _{K/\Q } (x \overline{y } )$.
Also, ${\mathcal O}_K $ is a Euclidean ring if $\mu (\mathcal O_K) < 1$.

\begin{prop}\label{dim2}
Assume that $\mu := \mu(\mathcal O_K) < 1$ and let $L$ be a $2$-dimensional Hermitian 
$\mathcal O_K$-lattice with $\min (L) = m$.
Then $d_L \geq m^2 (1-\mu) $. 
\end{prop}

\begin{proof}
The proof follows the argument of \cite[Lemma 4.2.2]{coul}. 
Let $x\in L$ be a minimal vector of $L$ and extend it to  an $\mathcal O_K$-basis of $L={\mathcal O}_K x + \mathcal O_K y $. 
Let $p(y) = bx $ denote the projection of $y$ onto $\langle x \rangle $.
Replacing $y$ by $y-ax $ with $a\in \mathcal O_K$ such that $N_{K/\Q }(a-b) \leq \mu $
we may assume that $N_{K/\Q }(b) = b \overline{b} \leq \mu $. Then
$$\begin{array}{lll}
d_L & = h(x,x) h(y-p(y),y-p(y)) & \geq h(x,x) (h(y,y) - \mu h(x,x) ) \\
& \geq (1-\mu ) h(x,x) h(y,y)  &
\geq (1-\mu ) m^2. \end{array} $$
\end{proof}

\begin{remark}\label{basis}
The proof shows that for $\mu < 1$ any 2-dimensional lattice $L$ has an $\mathcal O_K$-basis
$(x,y)$ such that 
$$h(x,x)h(y,y)(1-\mu) \leq d_L .$$
\end{remark}

The norm Euclidean imaginary quadratic number fields $\Q[\sqrt{-d}]$.
The last two lines give  the orbit representatives of the deep holes
under the action of $(\mathcal O_K^*) : \langle \overline{\phantom{a}} \rangle $

\begin{center}
\begin{tabular}{|c|c|c|c|c|c|}
\hline
d & 3 & 1 & 7 & 2 & 11 \\
\hline
$\mu $ & 1/3 & 1/2 & 4/7 & 3/4& 9/11 \\
\hline
$(1-\mu ) d_K $  & 2 & 2 & 3 & 2 & 2 \\
\hline
$\# $ deep holes & 6 & 4 & 6 & 4 & 6 \\
\hline
orbit repr. & $\frac{1}{\sqrt{-3}} $ & $\frac{1}{1+i}$ & $2/\sqrt{-7}$ & $\frac{1+\sqrt{-2}}{2}$& $3/\sqrt{-11}$ \\
of deep holes &  &  & $\frac{7+3\sqrt{-7}}{14} $ & & $\frac{11+5\sqrt{-11}}{22}$ \\
\hline
\end{tabular}
\end{center}

\begin{cor}
Let  $z\in K $ be a deep hole of $\mathcal O_K$. Then 
the lattice $L_K$ with Gram matrix 
$ \left( \begin{array}{cc} 1 & z \\ \overline{z}  & 1 \end{array} \right) $ 
the unique (up to $\mathcal O_K$-linear or antilinear isometry) 
densest $2$-dimensional Hermitian $\mathcal O_K$-lattice.
The $4$-dimensional $\Z $-lattice $(L_K,\Tr_{K/\Q} (h) )$ is isometric to
the root lattice $D_4$ for $d=3,1,2,11$ and to $A_2\perp A_2$ for $d=7$.
\end{cor}

This might give some hint of why tensor products of Hermitian lattices over 
$\Z[\frac{1+\sqrt{-7}}{2} ]$ seem to be more successful than over other 
rings of integers in imaginary quadratic fields. 

Also note that for $d=7$ and $d=11$, where there are 2 orbits of deep holes,
the corresponding lattices $L_K$ are isometric.

\section{Hermitian $\Z[\frac{1+\sqrt{-11}}{2} ]$-lattices.} 

We now apply the theory from above to the special case 
$K=\Q[\sqrt{-11}]$. Let $\eta := \frac{1+\sqrt{-11}}{2} $.
Then $\eta ^2-\eta +3 = 0$ and $\mathcal O_K = \Z[\eta ]$ 
 is an Euclidean domain with Euclidean minimum $\frac{9}{11}$. 

The Hermitian $\mathcal O_K$-structures of the Leech lattice have not 
been classified. However we may construct some of them using 
the classification of finite quaternionic matrix groups in \cite{quat} 
and embeddings of $K$ into definite quaternion algebras. 
It turns out that we obtain three different $\mathcal O_K$-structures, $P_1$, $P_2$ and $P_3$,
with automorphism groups $\Aut_{\mathcal O_K}(P_1) \cong 2.G_2(4)$ (with endomorphism algebra
${\mathcal Q}_{\infty,2}$),
 $\Aut_{\mathcal O_K}(P_2) \cong (L_2(7)\times \tilde{S}_3).2$ (with endomorphism algebra
${\mathcal Q}_{\infty,7}$),
 and
$\Aut_{\mathcal O_K}(P_3) \cong \SL_2(13).2$ (with endomorphism algebra 
${\mathcal Q}_{\infty,13}$). 

\begin{prop}
Let $T$ be the $2$-dimensional unimodular Hermitian $\mathcal O_K$-lattice with 
Gram matrix $\left( \begin{array}{cc} 2 & \eta \\ \overline{\eta } & 2 \end{array} \right) $. 
Let $(P,h) $ be some $12$-dimensional $\mathcal O_K $ lattice such that the trace lattice
$(P,\Tr _{K/\Q} \circ h )$ is isometric to the Leech lattice. 
Then the  Hermitian tensor product $R:=P\otimes _{\mathcal O_K} T$ has 
minimum either $2$ or $3$. 
The minimum of $R$ is $3$, if and only if $(P,h)$ does not represent 
one of the lattices $L_K$  or $T$. 
\end{prop}

\begin{proof}
The trace lattice of $R$ is an even unimodular lattice of dimension 48, so the Hermitian minimum
of $R$ is either 1, 2, or 3 and 
for any $v\in R$ we have $h(v,v)\in \Z $. So let $0\neq v  \in R$. In order to 
 apply Proposition \ref{bound} we need to deal with the two cases that the 
rank of $v$ is 1 or 2. If the rank of $v$ is 1, then $v = p\otimes t$ is a pure 
tensor and $h(v,v) \geq \min (P) \min (T) = 4$. 
If the rank of $v$ is 2, then by Proposition \ref{bound} 
$$h(v,v) \geq 2 d_2(P) ^{1/2}, \mbox{ because } d_2(T) = d_T = 1. $$ 
Since $d_2(P) \geq 2^2 (1-\mu ) = \frac{8}{11}$ the norm 
$h(v,v) \geq 2$ and $h(v,v) $ is strictly bigger than 2, if $d_2(P) > 1$. 
So let $L\leq P$ be a 2-dimensional sublattice of determinant $d_L \leq 1$. 
By Remark \ref{basis} the lattice $L$ has a basis $(x,y)$ such that 
$$(1-\mu) h(x,x) h(y,y) = \frac{2}{11} h(x,x) h(y,y) \leq d_L \leq 1 .$$
This  implies that $h(x,x) = h(y,y) = 2$ and the Gram matrix of $(x,y) $ is
$$\left(\begin{array}{cc} 2 & z \\ \overline{z} & 2 \end{array} \right) $$
for some $z\in \frac{1}{\sqrt{-11}} \mathcal O_K $. 
Since the minimum of $L$ is 2 and the densest 
2-dimensional $\mathcal O_K$-lattice of minimum 2 has determinant $\frac{8}{11}$ we obtain
$$4-z\overline{z} \in \{ \frac{8}{11} , \frac{9}{11}, \frac{10}{11} , 1 \} $$
There are no elements in $K$ with norm $\frac{35}{11}$ or $\frac{34}{11}$,
so the middle two possibilities are excluded. 
For the other two lattices we find $N(z) = z \overline{z} = 3 $ and then $L\cong T$
or $N(z) = \frac{36}{11} $ and $L\cong L_K$.
\end{proof}

\begin{cor}
$\min (P_1\otimes _{\mathcal O_K} T) = 2$ with kissing number $2\cdot 196560$,
$\min (P_2\otimes _{\mathcal O_K} T) = 2$ with kissing number $2\cdot 15120$,
 and 
$\min (P_3\otimes _{\mathcal O_K} T) = 3$.
The trace lattice of  the latter is isometric to the extremal even unimodular lattice 
$P_{48n}$ discovered in \cite{cycloquat}.
\end{cor}

\begin{proof}
For $P=P_1$, $P_2$, and $P_3$  we 
 computed orbit representatives of the $\Aut_{\mathcal O_K}(P)$-action on the set $S$ of minimal 
vectors of $P$. For each orbit representative $v$ we computed all inner products $h(v,w)$ with
$w\in S$ to obtain the representation number of $T$ and $L_K$ by $P$. \\
Let $P=P_1$.
Then $\mathcal M = \End _{\Aut_{\mathcal O_K}(P) }(P)  $ is the maximal order in 
the quaternion algebra ${\mathcal Q}_{\infty ,2}$.
 Given $v\in S$ there is a unique sublattice
$$\langle v \rangle _{\mathcal M} = \langle v,w \rangle _{\mathcal O_K} \cong _{\mathcal O_K} L_K.$$ 
The lattice $P_1$ does not represent the lattice $T$.
The lattice $P_2$ represents both lattices, $T$ and $L_K$, with multiplicity 
$10080$ and $5040$  respectively.
Only the lattice $P_3$ represents neither $T$ nor $L_K$.
\end{proof}

\section{Hermitian $\Z[\frac{1+\sqrt{-7}}{2} ]$-lattices.} 

We now restrict to the special case $K=\Q[\sqrt{-7}]$. 
Then $\mathcal O_K = \Z[\alpha ]$ 
where $\alpha ^2-\alpha +2 = 0$. Put $\beta := \overline{\alpha } = 1-\alpha $ 
its 
complex conjugate.
Then $\Z [\alpha ] $ is an Euclidean domain with Euclidean minimum $\frac{4}{7}$. 

Let $(P,h)$ be a Hermitian $\Z [\alpha ]$-lattice,
so $P$ is a free $\Z[\alpha ]$-module  and 
$h:P\times P \to \Q [\alpha ] $ a positive definite Hermitian form. 
One example of such a lattice is the
 {\em Barnes lattice}  $P_b$ with Hermitian Gram matrix 
$$\left( \begin{array}{ccc} 2 & \alpha & -1 \\  \beta & 2 & \alpha \\ -1 & \beta & 2 
\end{array} \right) $$
Then $P_b$ is Hermitian unimodular, 
$P_b = P_b^{\#}$ 
 and has Hermitian minimum 
$\min (P_b)= 2 $.

We will make use of the following two facts: 

{\bf Fact 1:} 
\begin{itemize}
\item[(a)] $d_1(P_b) = 2$.
\item[(b)] $d_2(P_b) = 2$. 
\item[(c)] $d_3(P_b) = d_{P_b} = 1$.
\end{itemize}

{\bf Fact 2:} 
\begin{itemize} 
\item[(a)]
By Proposition \ref{dim2} 
the unique densest $2$-dimensional $\Z[\alpha ]$-lattice is 
the lattice $P_a$ with
Gram matrix  $\left( \begin{array}{cc} 2 & 4/\sqrt{-7} \\ -4/\sqrt{-7} & 2 \end{array} \right) $,
$\min (P_a) = 2$, and $d_{P_a} = 12/7$.
\item[(b)]
There is a version of Voronoi theory also for Hermitian lattices developed in \cite{CoulVor}. 
This is used in the thesis \cite{Meyer} to classify the densest $\Z[\alpha ]$-lattices
in dimension 3. From this it follows that 
 $P_b$ is the globally
densest $3$-dimensional Hermitian $\Z [\alpha ]$-lattice.
\end{itemize} 

\begin{remark} 
The densest 8-dimensional 
$\Z$-lattice $E_8$ has a structure as a Hermitian $\Z[\alpha ]$-lattice $P_c$ of
dimension 4, which therefore realises the unique densest 4-dimensional $\Z[\alpha ]$-lattice. 
\end{remark}

From the two facts above we immediately obtain the following Proposition. 

\begin{prop} \label{densest} 
 Let $(P,h)$ be a Hermitian $\Z[\alpha ]$
lattice of dimension $\geq 3$ and with $\min (P) =: m$.  Then  
\begin{itemize}
 \item[(a)] $d_1(P) = \min (P) = m $.
\item[(b)] $d_2(P) \geq \frac{3m^2}{7} $.
\item[(c)] $d_3(P) \geq \frac{m^3}{8}$ and $d_3(P) = \frac{m^3}{8}$ if and only if $P$ contains a sublattice
isometric to $\sqrt{m/2} P_b$.
\end{itemize}
\end{prop}

\subsection{An application to unimodular 72-dimensional lattices.}\label{Herm}

We now apply the theory from the previous sections to obtain a new
proof for the extremality of the even unimodular lattice $\Gamma $ in 
dimension 72 from \cite{dim72}.
 Michael Hentschel 
 \cite{Hentschel} 
 classified all Hermitian $\Z[\alpha ]$-structures 
on the 
even unimodular $\Z $-lattices of dimension 24 using the 
Kneser neighbouring method \cite{Kneser} to generate the lattices 
and checking completeness with the mass formula.
In particular there are exactly nine such $\Z[\alpha ]$ structures 
$(P_i,h)$  ($1\leq i \leq 9$)
such that the trace lattice
$(P_i,  \Tr _{\Z[\alpha ]/\Z } \circ h) \cong \Lambda $ is
the Leech lattice.
They are used by the second author in 
\cite{dim72} to construct  nine 36-dimensional 
Hermitian $\Z[\alpha ]$-lattice $R_i $  defined by 
$(R_i,h) := P_b \otimes _{\Z[\alpha ]} P_i $.
Using the methods described above we obtain the following main 
result on the minimum of these tensor products. 

\begin{theorem} \label{main} 
The minimum of the Hermitian lattices $R_i$ is either $3$ or $4$.
The number of vectors of norm $3$ in $R_i$ is equal to the representation 
number of $P_i$ for the sublattice $P_b$. 
In particular $\min (R_i) = 4$ if and only if the Hermitian Leech lattice $P_i$ 
does not contain a sublattice isomorphic to $P_b$.
\end{theorem}

\begin{proof}
The proof follows from Proposition \ref{bound} and uses Proposition \ref{densest}:
(An alternative proof that is not based on the computation of perfect $\Z[\alpha ]$-lattices
is given in the next section.) \\
Let $z\in P_i \otimes _{\Z[\alpha ]} P_b$ be a non-zero vector of rank $r=1,2,$ or $3$.
\begin{itemize}
\item
If $r=1$, then $z=v\otimes w $ and $h(z,z) \geq \min(P_i) \min(P_b) = 4$. 
\item
If $r=2$, then $h(z,z) \geq 2 \sqrt{d_2(P_b)} \sqrt{d_2(P_i)} 
\geq 2 \sqrt{2} \sqrt{\frac{12}{7}} > 3$, so $h(z,z) \geq 4$.
\item
If $r=3$, then  $h(z,z) \geq 3 d_3(P_i)^{1/3} \geq 3$. 
Since $h(z,z) \in \Z $ this implies that $h(z,z) \geq 3$ and equality requires that $P_i$ 
contains a minimal section isometric to $P_b^{\#}=P_b$. 
\end{itemize}
\end{proof}

\begin{cor} 
Let $P_1$ denote the Hermitian Leech lattice with automorphism group $\SL_2(25)$ 
(see \cite{dim72}). 
Then $\min (P_1 \otimes _{\Z[\alpha ]} P_b) = 4$.  
For the  other eight Hermitian Leech lattices  $P_i$ the minimum is
$\min (P_i \otimes _{\Z[\alpha ]} P_b) = 3 $ ($i=2,\ldots ,9$). 
\end{cor}

\begin{proof}
With MAGMA (\cite{MAGMA}) we computed the  number of sublattices isomorphic to $P_b$ in 
the lattices $P_i$. Only one of them, $P_1$,
 does not contain such a sublattice, so 
$d_3(P_1) > 1$ and hence $\min (P_1 \otimes _{\Z[\alpha ]} P_b) \geq 4$.  
For the computation we went through orbit representatives $v_1$ of the 
Hermitian automorphism group $\Aut(P_i)$ on the set $S$ of minimal vectors of the
Leech lattice. For any $v_1$ we compute the set 
$$A(v_1) := \{ v\in S \mid h(v,v_1) = \alpha \} .$$
In all cases this set $A(v_1)$ has 32 elements. 
For all $v_2 \in A(v_1)$ we count the number of vectors $v\in S$ such that 
$h(v,v_2) = \alpha $ and $h(v,v_1) = -1 $. 
This computation takes about 30 seconds per orbit representative $v_1$. 
\end{proof}

\subsection{An alternative proof of  Theorem \ref{main}} 

The thesis \cite{Meyer} uses the Voronoi algorithm to compute the 3-dimensional
perfect $\Z[\alpha ]$-lattices. The proof of Theorem \ref{main} only uses the 
following proposition which can be proved without computer. 

\begin{prop} Let $P$ be one of the nine $\Z[\alpha ]$-structures of the 
Leech lattices $\Lambda _{24}$.   
Then  
\begin{itemize}
 \item[(a)] $d_1(P) = \min (P) = 2 $. 
\item[(b)] $d_2(P) = \frac{12}{7} $. 
\item[(c)] $d_3(P) \geq 1$.
\end{itemize}
\end{prop}

\begin{proof}
(a) follows from the fact that the Leech lattice is extremal.  \\
(b) By Proposition \ref{dim2}  the discriminant 
$d_M$ 
of a $\Z[\alpha ]$-lattice $M$ of rank 2 satisfies 
$$d_M \geq \frac{3}{7} \min (M) ^2 .$$ 
If $M$ is a sublattice of $P$, then $\min (M) \geq 2$ and hence $d_M
\geq \frac{12}{7} $. On the other hand all nine Hermitian structures contain sublattices $P_a$
of determinant $\frac{12}{7}$.
\\
(c) Assume by way of contradiction that $d_3(P) < 1$.
 Since $P^{\#}=\sqrt{-7}P$, we have $h(x,y) \in \dfrac{1}{\sqrt{-7}}\Z [\alpha ] $ 
for any $x,y$ in $P$, and moreover, since $P$ is even as a Euclidean lattice,
 we see that $h(x,x) \in \Z$ for $x \in P$.
 As a consequence, if $M=\oplus_{i=1}^3 \Z [\alpha ]  e_i$ is a $3$-dimensional section of $P$,
  its discriminant $d_M \colon = \det(h(e_i,e_j))$ belongs to $\dfrac{1}{7}\Z$.
 In particular
$$d_M < 1 \Longrightarrow d_M \leq \dfrac{6}{7}.$$ 
Furthermore, $\gamma_h(M)\colon =\dfrac{\min M}{d_M^{1/3}}$ is bounded from above (see \cite{coul}) by 
\begin{equation}\label{boundeq} 
\gamma_h(M) \leq \dfrac{\sqrt{7}}{2}\gamma_6=\dfrac{\sqrt{7}}{3^{1/6}}\simeq 2.203
\end{equation}
which immediately implies that $d_M \geq \dfrac{8\sqrt{3}}{7\sqrt{7}}>5/7$. We conclude that
$$d_M < 1 \Longrightarrow d_M = \dfrac{6}{7}.$$

Next we show that if such a sublattice $M$ with $d_M = \dfrac{6}{7}$ exists, 
then it admits a minimal $2$-dimensional subsection generated over $\Z [\alpha ] $ by two minimal vectors of $P$.
 Otherwise we would have, by Remark \ref{basis}, 
$$d_2(M) \geq \frac{3}{7} 3 \cdot 2 =\dfrac{18}{7} $$
whence, using the identity 
$d_M=d_2(M)d_1(M^{\#})^{-1}$ (see Lemma \ref{ddd}),
$$\gamma_h(M^{\#}) \geq \dfrac{d_2(M)}{d_M^{2/3}}\geq 
3\left( \dfrac{6}{7}\right)^{1/3}\simeq 2.85$$ 
violating bound (\ref{boundeq}). 

Thus, one can find a $\Z [\alpha] $-basis 
$\left\lbrace e_1,e_2,e_3\right\rbrace $ of $M$, such that $h(e_1,e_1)=h(e_2,e_2)=2$ 
and  $M_2 \colon \!\!\!= \Z [\alpha ] e_1 \oplus \Z [\alpha ] e_2$
 is a minimal $2$-dimensional section of $M$. Setting $h(e_1,e_2)=\dfrac{a}{\sqrt{-7}}$,
 with $a \in \Z [\alpha ]$ we see that
$$\dfrac{12}{7} = d_2(P)\leq \det \begin{pmatrix} 2 & \dfrac{a}{\sqrt{-7}}\\ -\dfrac{\overline{a}}{\sqrt{-7}}&2\end{pmatrix} = d_2(M) \leq \gamma_h(M^{\#})d_M^{2/3} \leq \dfrac{\sqrt{7}}{3^{1/6}}\left( \dfrac{6}{7}\right)^{2/3}\simeq 1.988$$
which yields $14 < a \overline{a} \leq 16$, whence $a \overline{a} = 16$ ($15$ is not a norm),
 and $d_2(M)=d_2(P)=\dfrac{12}{7}$.
 Replacing $e_2$ by $\pm e_2\pm e_1$ if necessary, we may finally assume that $h(e_1,e_2)=\dfrac{4}{\sqrt{-7}}$.
Finally, we have the formula
$$\dfrac{6}{7}=d_M=d_{M_2} h\left( q(e_3),q(e_3)\right) =d_{M_2}\left( h(e_3,e_3) - h\left( p(e_3),p(e_3)\right) \right) $$
where $p$ and $q$ stand respectively for the orthogonal projection on the subspace
 $F \colon \!\!\!=\Q [\alpha ] M_2=\Q [\alpha ] e_1 + \Q [\alpha ] e_2$ and its orthogonal complement
 $F^{\perp}$ (see Lemma \ref{proj}). 
Furthermore, we may replace $e_3$ by $e_3+u$, with $u\in M_2$, 
and it is easily seen that $u$ may be chosen so that $h(p(e_3+u),p(e_3+u)) \leq \dfrac{80}{49}$
 (the Hermitian norm of any vector $v=xe_1+ye_2$ in $F$ is given by
 $h(v,v)=\dfrac{2}{7}\left( 7\vert x + \dfrac{2}{\sqrt{-7}}y \vert^2 + 3 \vert y \vert^2\right)$,
 and since $\Z[\alpha]$ is Euclidean with Euclidean minimum $\dfrac{4}{7}$
 we may choose $y'$ and $x'$ in $\Z[\alpha]$ such that $\vert y-y' \vert^2 \leq \dfrac{4}{7}$ and
 $\vert (x-x') + \dfrac{2}{\sqrt{-7}}(y-y') \vert^2 \leq \dfrac{4}{7}$, whence the conclusion).
 Consequently, one has 
$$\dfrac{6}{7}=d_M \geq \dfrac{12}{7}\left( h(e_3,e_3) - \dfrac{80}{49} \right)$$
which implies that $h(e_3,e_3)=2$.

Finally, the Hermitian Gram matrix of $M$ is
$$\begin{pmatrix} 2 & 4/\sqrt{-7}&a/\sqrt{-7}\\-4/\sqrt{-7}&2&b/\sqrt{-7}\\-\overline{a}/\sqrt{-7}&-\overline{b}\sqrt{-7}&2
\end{pmatrix}$$
with $a, b$ in $\Z[\alpha]$, of norm at most $16$
 (this is because the determinant of any $2$-dimensional section is at least $12/7$).
 Consequently, there are finitely many possible $a$ and $b$, and it is not hard to check that,
 up to permutation of $e_1$ and $e_2$ and sign change for $e_3$,
 the only choice to achieve the condition $d_M =6/7$ 
is $a=3/\sqrt{-7}$ and $b=0$.
 But this leads to a Hermitian Gram matrix $\begin{pmatrix} 2 & 4/\sqrt{-7}&3/\sqrt{-7}\\-4/\sqrt{-7}&2&0\\-3/\sqrt{-7}&0&2
\end{pmatrix}$ of minimum $1$, a contradiction.
\end{proof}

\end{document}